\documentclass[11pt]{article}
\usepackage{amsfonts}

\usepackage{pgf}
\usepackage[square, comma, sort&compress, numbers]{natbib}
\usepackage[english]{babel}
\usepackage[utf8]{inputenc}
\usepackage{datetime}
\usepackage{geometry}
\usepackage{array}
\usepackage{xspace}

\usepackage{amsmath,amsfonts,amssymb,amsopn,amsthm}
\theoremstyle{plain}
\usepackage{graphicx} 
\usepackage{epstopdf}
\usepackage{enumerate}

\def\Box{\vcenter{\vbox{\hrule\hbox{\vrule
     \vbox to 8.8pt{\hbox to 10pt{}\vfill}\vrule}\hrule}}}
\usepackage{indentfirst}
\newtheorem{theorem}{Theorem}[section]

\newtheorem{corollary}[theorem]{Corollary}

\newtheorem{definition}[theorem]{Definition}
\newtheorem{example}[theorem]{Example}

\newtheorem{lemma}[theorem]{Lemma}

\large\normalsize

\begin{document}

%

\title{A new generalized inverse of matrices from core-EP decomposition
}

\author{Kezheng Zuo$^{1}$, Yu Li$^{2}$, Gaojun Luo$^{3}$}

\maketitle

\let\thefootnote\relax\footnotetext{$^1$Department of Mathematics, Hubei Normal
University, Hubei, Huangshi, China. (Email: xiangzuo28@163.com)}

\let\thefootnote\relax\footnotetext{$^2$Department of Mathematics, Hubei Normal
University, Hubei, Huangshi, China. (Email: 335972971@qq.com)}

\let\thefootnote\relax\footnotetext{$^3$Department of Mathematics, Nanjing University of Aeronautics and Astronautics, Nanjing 211100, China. (Email: gjluo1990@163.com)}

\let\thefootnote\relax\footnotetext {This work was supported by NSFC of China (Grant No. 11961076). }

\begin{abstract}
A new generalized inverse for a square matrix $H\in\mathbb{C}^{n\times n}$, called CCE-inverse, is established by the core-EP decomposition and Moore-Penrose inverse $H^{\dag}$. We propose some characterizations of the CCE-inverse. Furthermore, two canonical forms of the CCE-inverse are presented. At last, we introduce the definitions of CCE-matrices and $k$-CCE matrices, and prove that CCE-matrices are the same as $i$-EP matrices studied by Wang and Liu in [The weak group matrix, Aequationes Mathematicae, 93(6): 1261-1273, 2019].

\end{abstract}

\section{Introduction}\numberwithin{equation}{section}

Generalized inverses of matrices are closely associated with orthonormalization, linear equations, singular values, least squares solutions and various matrix factorizations. Over the past few decades, there has been an increasing interest in the study of generalized inverses due to their wide utilization in many fields such as statistics \cite{A1}, neural network \cite{A2}, compressed sensing \cite{A3} and so on.

Let $\mathbb{C}^{m\times n}$ stand for the set of $m\times n$ complex matrices. Denote the range space, null space, conjugate transpose and rank of $H\in\mathbb{C}^{m\times n}$ by $\mathcal{R}(H)$, $\mathcal{N}(H)$, $H^{*}$ and $r(H)$, respectively. Moreover $I_n$ will be the identity matrix of
order $n$. For a matrix $H\in\mathbb{C}^{n\times n}$, the index of $H$ is said to be the smallest integer $k>0$ such that $r(H^{k})=r(H^{k+1})$ and written ${\rm Ind}(H)$. Let $\mathbb{C}^{n\times n}_{k}$ be the set consisting of $n\times n$ complex matrices with index $k$. For a matrix $H\in\mathbb{C}^{m\times n}$, if there exists a unique matrix $X\in\mathbb{C}^{n\times m}$ satisfying
$$
(1) \ HXH=H, \ \ \ (2) \ XHX=X, \ \ \ (3) \ (HX)^{*}=HX, \ \ \ (4) \ (XH)^{*}=XH,
$$
then $X=H^{\dag}$ is termed Moore-Penrose (MP for short) inverse of $H$ \cite{P}. Furthermore, a matrix $X\in\mathbb{C}^{n\times m}$ satisfying the condition (1) is called a $g$-inverse of $H$. A matrix $X\in\mathbb{C}^{n\times m}$ is said to be an outer inverse of $H$ if $XHX=X$ and is denoted by $H^{(2)}$. If there exists a matrix $X\in\mathbb{C}^{n\times m}$ that satisfies the conditions (1) and (2), then it is referred to as a reflexive g-inverse of $H$.  A matrix $X$ satisfying $X=XHX$, $\mathcal{R}(H)=\mathcal{T}$ and $\mathcal{N}(H)=\mathcal{S}$ is denoted by $A^{(2)}_{\mathcal{T},\mathcal{S}}$. Basing on the MP-inverse, one can obtain the orthogonal projection onto $\mathcal{R}(H)$, which is represented by $P_H=HH^{\dag}$.

The Drazin inverse $H^{D}\in\mathbb{C}^{n\times n}$ of $H\in\mathbb{C}^{n\times n}_{k}$ is the unique solution of the following equations \cite{D0}:
$$
HH^{D}H=H, \ \ \  HH^{D}=H^{D}H, \ \ \ H^{D}H^{k+1}=H^{k}.
$$
Especially, the Drazin inverse of $H$ is reduced to the group inverse of $H$ which is denoted by $H^{\#}$ if ${\rm Ind}(H)=1$. In general, nonsingular matrices and matrices with index $1$ have nice algebraic structure. Assume that $\mathbb{C}^{\mathrm{CM}}_{n}$ is the set consisting of matrices with order $n$ and index less than or equal to $1$. For any matrix $H\in\mathbb{C}^{\mathrm{CM}}_{n}$, the authors \cite{Bt2} established a new generalized inverse, called core inverse, which is the unique solution of
$$
HX=P_{H},\ \ \  \mathcal{R}(X)\subseteq\mathcal{R}(H).
$$
Denote by $H^{\textcircled{\#}}$ the core inverse of $H\in\mathbb{C}^{\mathrm{CM}}_{n}$. The core inverse has been an important concept in the study of the generalized inverses (see \cite{Bt2,K,LZZ,RD}). Moreover, several generalizations of the core inverse, i.e., core-EP inverse, DMP-inverse, BT-inverse, $(B,C)$-inverse, weak group inverse and CMP-inverse, were provided by the unique solution of several equations as follows.

For a matrix $H\in\mathbb{C}^{n\times n}_{k}$, the unique solution $H^{\textcircled{\emph{\dag}}}\in\mathbb{C}^{n\times n}$ of the equations
$$
H^{\textcircled{\emph{\dag}}}HH^{\textcircled{\emph{\dag}}}=H^{\textcircled{\emph{\dag}}},\ \ \  \mathcal{R}(H^{\textcircled{\emph{\dag}}})=\mathcal{R}((H^{\textcircled{\emph{\dag}}})^{*})=\mathcal{R}(H^{k}),
$$
is called the core-EP inverse of $H$ ([see \cite{PM,flt2,w,ZC}).

The DMP-inverse \cite{MT,ZDC} $H^{D,\dag}$ of $H\in\mathbb{C}^{n\times n}_{k}$ is defined as
$$
H^{D,\dag}HH^{D,\dag}=H^{D,\dag}, \ \ \  H^{D,\dag}H=H^{D}H, \ \ \ H^{k}H^{D,\dag}=H^{k}H^{\dag}.
$$
Moreover, it was demonstrated that $H^{D,\dag}=H^{D}HH^{\dag}$. Also, the dual DMP-inverse of $H$ is defined to be the matrix $H^{\dag,D}=H^{\dag}HH^{D}$ \cite{MT}.

Baksalary et al. \cite{Bt1} introduced the BT-inverse $H^{\diamondsuit}$ of $H\in\mathbb{C}^{n\times n}$, which is defined by
$$
H^{\diamondsuit}=(H^{2}H^{\dag})^{\dag}=(HP_{H})^{\dag}.
$$

The $(B,C)$-inverse \cite{BBJ,D} of $H\in\mathbb{C}^{m\times n}$, denoted by $H^{(B,C)}$, is the unique matrix $X\in\mathbb{C}^{n\times m}$ such that
$$
   CHX=C, \ \ \  XHB=B, \ \ \  \mathcal{R}(X)=\mathcal{R}(B), \ \ \ \mathcal{N}(X)=\mathcal{N}(C),
$$
where $B, C\in\mathbb{C}^{n\times m}$.

In 2018, Wang and Chen \cite{wc} defined the weak group inverse $H^{\textcircled{w}}$ of $H\in\mathbb{C}^{n\times n}_{k}$ by
$$
(H^{\textcircled{w}})^{2}=H^{\textcircled{w}}, \ \ \ HH^{\textcircled{w}}=H^{\textcircled{\dag}}H.
$$

Very recently, Mehdipour and Salem \cite{MS} proposed a new generalized inverse, termed the CMP-inverse of $H\in\mathbb{C}^{n\times n}_{k}$, written as $H^{C,\dag}$, which is defined by
$$
H^{C,\dag}HH^{C,\dag}=H^{C,\dag}, \ \ \  HH^{C,\dag}H=\widetilde{H}_{1}, \ \ \  HH^{C,\dag}=\widetilde{H}_{1}H^{\dag}, \ \ \ H^{C,\dag}H=H^{\dag}\widetilde{H}_{1},
$$
where the matrix $\widetilde{H}_{1}$ is the core part of core-nilpotent decomposition of $H$ \big(in fact, $\widetilde{H}_{1}=HH^{D}H$, see \cite{MS,XCM}\big).

Note that the CMP-inverse is derived from the core-nilpotent decomposition and MP-inverse. Wang \cite{w} proposed a new decomposition of $H\in\mathbb{C}^{n\times n}_k$, which is referred to as core-EP (CEP for simplicity) decomposition. It is natural to consider the CEP-decomposition in the definition of the CMP-inverse. As a result, we get a new generalized inverse called the CCE-inverse. In this work, we introduce the CCE-inverse $H^{C,\textcircled{\dag}}$ for square matrices of an arbitrary index using its core part $H_{1}$ from the CEP-decomposition of $H$ and its MP-inverse $H^{\dag}$. Using the CEP-decomposition of $H$, we derive some characterizations of the CCE-inverse of $H$. Meanwhile, we introduce CCE-matrices and $k$-CCE matrices and show that these two kinds of matrices are consistent. Finally, we study the relationships between CCE-matrices and $k$-CCE matrices with some special matrix classes.

The rest of the material is organized as follows. In Section 2, a new generalized inverse is established and we study its characterizations. In Section 3, two canonical forms of CCE-inverse and their applications are presented. In Section 4, we investigate CCE-matrices and $k$-CCE matrices.

\section{A new generalized inverse on core-EP decomposition}

In this section, we propose a new generalized inverse of $H$ by the CEP-decomposition and MP-inverse. We begin with the CEP-decomposition.

Wang \cite{w} introduced the CEP-decomposition of $H\in\mathbb{C}^{n\times n}_k$, which says that a matrix $H\in\mathbb{C}^{n\times n}_k$ can be represented in the following form:
\begin{equation}\label{eq2.5}
H=H_{1}+H_{2}=U \left[\begin{array}{cc}
T & P \\
0 & Q \\
\end{array}
\right] U^{*},\ \ \ H_{1}=U \left[\begin{array}{cc}
T & P \\
0 & 0 \\
\end{array}
\right] U^{*},\ \ \ H_{2}=U \left[\begin{array}{cc}
0 & 0 \\
0 & Q \\
\end{array}
\right] U^{*},
\end{equation}
where $T$ is nonsingular with $r(T)=r(H^{k})$ and $Q$ is nilpotent with index $k$ and $U\in\mathbb{C}^{n\times n}$ is unitary. The expression of $H$ provided by $(\ref{eq2.5})$ is unique and satisfies ${\rm Ind}(H_{1})\leq1$, $H_{2}^{k}=0$ and $H_{1}^{*}H_{2}=H_{2}H_{1}=0$ \cite[Theorem 2.1]{w}. In $(\ref{eq2.5})$, $H_{1}$ and $H_{2}$ are termed the core part and nilpotent part of $H$, respectively. In addition, it is easy to verify that $H_{1}=HH^{\textcircled{\dag}}H$. For simplicity, we always assume that the matrix $H\in\mathbb{C}^{n\times n}_{k}$ has the CEP-decomposition $H=H_{1}+H_{2}$ throughout this section.

\begin{definition}
Suppose that $H\in\mathbb{C}^{n\times n}_{k}$. Then the CCE-inverse of $H$ is defined by $H^{C,\textcircled{\dag}}=H^{\dag}H_{1}H^{\dag}=H^{\dag}HH^{\textcircled{\dag}}HH^{\dag}=Q_{H}H^{\textcircled{\dag}}P_{H}$,
where $P_{H}=HH^{\dag}$ and $Q_{H}=H^{\dag}H$.
\end{definition}

\begin{theorem}\label{thm3.1}
Assume that $H\in\mathbb{C}^{n\times n}_{k}$. Then $X=H^{C,\textcircled{\dag}}$ is the unique solution of the following three equations
\begin{equation}\label{eq3.1}
XHX=X, \ \ \ HX=H_{1}H^{\dag}, \ \ \ XH=H^{\dag}H_{1}.
\end{equation}
\end{theorem}
\begin{proof}
Obviously, $H^{C,\textcircled{\dag}}=H^{\dag}H_{1}H^{\dag}$ is a solution of this system. For uniqueness, suppose that $X$ satisfies $(\ref{eq3.1})$. Then
\begin{eqnarray*}
X&=&XHX\\
&=&X(H_{1}H^{\dag})\\
&=&X(HH^{\textcircled{\dag}}H)H^{\dag}\\
&=&(H^{\dag}H_{1})(H^{\textcircled{\dag}}HH^{\dag})\\
&=&H^{\dag}H_{1}H^{\dag}.
\end{eqnarray*}
Hence, $X=H^{C,\textcircled{\dag}}$ is the unique matrix such that $XHX=X, HX=H_{1}H^{\dag}, XH=H^{\dag}H_{1}$.
\end{proof}

In the following, we use a concrete sample to illustrate the difference between the CCE-inverse and other generalized inverses.
\begin{example}\label{ex1}
Let $H=\left[\begin{array}{cc}
I_{3} & I_{3} \\
0 & N \\
\end{array}
\right],$
where
$N=\left[\begin{array}{ccc}
0 & 1 & 0\\
0 & 0 & 1\\
0 & 0 & 0\\
\end{array}
\right]$. It is easy to check that $Ind(H)=3$. By a direct computation, we obtain the generalized inverses as follows:
$$
H^{\dag}=\left[\begin{array}{cc}
H_{1} & -N^{\dag} \\
I-H_{1} & N^{\dag} \\
\end{array}
\right], \;\
H^{D}=\left[\begin{array}{cc}
I_{3} & H_{2} \\
0 & 0 \\
\end{array}
\right],
$$

$$
H^{D,\dag}=H^{D}HH^{\dag}=\left[\begin{array}{cc}
I & H_{3} \\
0 & 0 \\
\end{array}
\right], \;\
H^{\dag,D}=H^{\dag}HH^{D}=\left[\begin{array}{cc}
H_{1} & H_{4} \\
I-H_{1} & H_{2}-H_{4} \\
\end{array}
\right],
$$

$$
H^{\diamond}=(H^{2}H^{\dag})^{\dag}=\left[\begin{array}{cc}
H_{5} & -H_{6} \\
I-H_{5} & H_{6} \\
\end{array}
\right], \;\
H^{\textcircled{\dag}}=\left[\begin{array}{cc}
I_{3} & 0 \\
0 & 0 \\
\end{array}
\right],
$$

$$
H^{C,\dag}=Q_{H}H^{D}P_{H}=\left[\begin{array}{cc}
H_{1} & H_{7} \\
I-H_{1} & H_{3}-H_{7} \\
\end{array}
\right], \;\
H^{\textcircled{w}}=\left[\begin{array}{cc}
I_{3} & I_{3} \\
0 & 0 \\
\end{array}
\right].
$$
It is easy to see that the CCE-inverse $H^{C,\textcircled{\dag}}$ is
$$H^{C,\textcircled{\dag}}=\left[\begin{array}{cc}
H_{1} & 0 \\
I-H_{1} & 0 \\
\end{array}
\right],$$
where $H_{1}=\left[\begin{array}{ccc}
\frac{1}{2} & 0& 0 \\
0 & 1& 0 \\
0 & 0& 1 \\
\end{array}
\right], \;\ H_{2}=\left[\begin{array}{ccc}
1 & 1& 1 \\
0 & 1& 1 \\
0 & 0& 1 \\
\end{array}
\right], \;\ H_{3}=\left[\begin{array}{ccc}
1 & 1& 0 \\
0 & 1& 0 \\
0 & 0& 0 \\
\end{array}
\right], \;\ H_{4}=\left[\begin{array}{ccc}
\frac{1}{2} & \frac{1}{2}& \frac{1}{2} \\
0 & 1& 1 \\
0 & 0& 1 \\
\end{array}
\right],$ \\ $H_{5}=\left[\begin{array}{ccc}
1 & 0& 0 \\
0 & \frac{1}{2}& 0 \\
0 & 0& 1 \\
\end{array}
\right], \;\ H_{6}=\left[\begin{array}{ccc}
0 & 0& 0\\
\frac{1}{2} & 0& 0 \\
0 & 0& 0 \\
\end{array}
\right], \;\ H_{7}=\left[\begin{array}{ccc}
\frac{1}{2} & \frac{1}{2}& 0 \\
0 & 1& 0 \\
0 & 0& 0 \\
\end{array}
\right]$ and \;\ $N^{\dag}=\left[\begin{array}{ccc}
0 & 0& 0 \\
1 & 0& 0 \\
0 & 1& 0 \\
\end{array}
\right].$
\end{example}

According to Example \ref{ex1}, the CCE-inverse is indeed a new generalized inverse. Now, we show that the CCE-inverse $H^{C,\textcircled{\dag}}$ of $H$ is an outer inverse of $H$ (i.e.,$H^{C,\textcircled{\dag}}HH^{C,\textcircled{\dag}}=H^{C,\textcircled{\dag}}$) and a reflexive g-inverse of $H_{1}$.

\begin{theorem}
Let $H\in\mathbb{C}^{n\times n}_{k}$ with CEP-decomposition $H=H_{1}+H_{2}$ defined by $(\ref{eq2.5})$. Then
\begin{itemize}
\item[(a)] $H^{C,\textcircled{\dag}}$ is an outer inverse of $H$;
\item[(b)] $H^{C,\textcircled{\dag}}$ is a reflexive g-inverse of $H_{1}$.
\end{itemize}
\end{theorem}
\begin{proof}
(a).  This is evident.

(b).  To prove the result, we need to show
$$
H_{1}H^{C,\textcircled{\dag}}H_{1}=H_{1}, \ \ \ H^{C,\textcircled{\dag}}H_{1}H^{C,\textcircled{\dag}}=H^{C,\textcircled{\dag}}.
$$
Using the properties of $H^{\dag}$ and $H^{\textcircled{\dag}}$, it is easy to prove both of two equations above.
Hence, the result holds.
\end{proof}
\begin{theorem}
Let $H\in\mathbb{C}^{n\times n}_{k}$. Then the following assertions are equivalent:
\begin{itemize}
\item[(a)] $H\in\mathbb{C}^{\mathrm{CM}}_{n}$;
\item[(b)] $H^{C,\textcircled{\dag}}\in H\{1\}\ (i.e.,HH^{C,\textcircled{\dag}}H=H)$;
\item[(c)] $H^{C,\textcircled{\dag}}=H^{\dag}$.
\end{itemize}
\end{theorem}
\begin{proof}
``$(a)\Leftrightarrow(b)$''. It can be easily seen that
$A^{C,\textcircled{\dag}}\in H\{1\}$ if and only if $H_{2}=0$ which is equivalent to $N=0$, i.e. $r(H)=r(H^{2})$.

``$(b)\Leftrightarrow(c)$''. Premultiplying and postmultiplying $H^{C,\textcircled{\dag}}=H^{\dag}$ by $H$, we obtain $H^{C,\textcircled{\dag}}\in H\{1\}$. Premultiplying and postmultiplying $HH^{C,\textcircled{\dag}}H=H$ by $H^{\dag}$ we have $H^{C,\textcircled{\dag}}=H^{\dag}$. The the desired result follows.
\end{proof}
\begin{theorem}\label{thm3.6}
Let $H\in\mathbb{C}^{n\times n}_{k}$. Then
$H^{C,\textcircled{\dag}}=H^{\dag}P_{H^{k}}=H^{\dag,D}P_{H^{k}}=H^{C,\dag}P_{H^{k}}$.
\end{theorem}
\begin{proof}
Since $HH^{\textcircled{\dag}}=H^{k}(H^{k})^{\dag}$ (see Corollary 3.4 of \cite{w}), we have
$$
H^{C,\textcircled{\dag}}=H^{\dag}(HH^{\textcircled{\dag}})HH^{\dag}=H^{\dag}P_{H^{k}}P_{H}=H^{\dag}P_{H^{k}}.
$$
Meanwhile, we have  $H^{\dag,D}P_{H^{k}}=H^{\dag}HH^{D}H^{k}(H^{k})^{\dag}=H^{\dag}H^{k}(H^{k})^{\dag}=H^{\dag}P_{H^{k}}=H^{C,\textcircled{\dag}}.$\\
Observe that $H^{C,\dag}P_{H^{k}}=H^{\dag}(HH^{D}H)H^{\dag}P_{H^{k}}=H^{\dag}HH^{D}P_{H^{k}}=H^{\dag,D}P_{H^{k}}=H^{C,\textcircled{\dag}}$.
Therefore, $H^{C,\textcircled{\dag}}=H^{\dag}P_{H^{k}}=H^{\dag,D}P_{H^{k}}=H^{C,\dag}P_{H^{k}}$.
\end{proof}
\begin{corollary}\label{cor3.7}
Let $H\in\mathbb{C}^{n\times n}_{k}$. Then
$$r(H^{C,\textcircled{\dag}})=r(H^{k}), \ \ \ \mathcal{R}(H^{C,\textcircled{\dag}})=\mathcal{R}(H^{\dag}H^{k}) \ \ \ and \ \ \ \mathcal{N}(H^{C,\textcircled{\dag}})=\mathcal{N}((H^{k})^{*}).$$
\end{corollary}
\begin{proof}
By Theorem \ref{thm3.6}, we deduce that
$$\mathcal{R}(H^{C,\textcircled{\dag}})=\mathcal{R}(H^{\dag}P_{H^{k}})=H^{\dag}\mathcal{R}(P_{H^{k}})=H^{\dag}\mathcal{R}(H^{k})=\mathcal{R}(H^{\dag}H^{k}),$$
$$\mathcal{N}(H^{C,\textcircled{\dag}})=\mathcal{N}(H^{\dag}P_{H^{k}})=\mathcal{N}(P_{H^{k}})=\mathcal{N}((H^{k})^{*}).$$
Due to $\mathcal{R}(H^{C,\textcircled{\dag}})=\mathcal{R}(H^{\dag}H^{k})$, we get that $r(H^{C,\textcircled{\dag}})=r(H^{\dag}H^{k})=r(H^{k}).$
\end{proof}
\begin{theorem}
Let $H\in\mathbb{C}^{n\times n}_{k}$. Then
\begin{itemize}
\item[(a)] $HH^{C,\textcircled{\dag}}$ is an orthogonal projection onto $\mathcal{R}(H^{k})$, i.e., $HH^{C,\textcircled{\dag}}=P_{H^{k}}$;
\item[(b)] $H^{C,\textcircled{\dag}}H$ is a projector onto $\mathcal{R}(H^{\dag}H^{k})$ along $\mathcal{N}((H^{k})^{\dag}H)$, i.e., $H^{\textcircled{\dag}}H=P_{\mathcal{R}(H^{\dag}H^{k}),\mathcal{N}((H^{k})^{\dag}H)}$.
\end{itemize}
\end{theorem}

\begin{proof}
(a). By Theorem \ref{thm3.6}, we deduce that $HH^{C,\textcircled{\dag}}=HH^{\dag}P_{H^{k}}=P_{H^{k}}$.

(b). Since $H^{C,\textcircled{\dag}}$ is an outer inverse of $H$, $H^{C,\textcircled{\dag}}H$ is a projector. It follows from
$
\mathcal{R}(H^{C,\textcircled{\dag}}H)=\mathcal{R}(H^{C,\textcircled{\dag}})=\mathcal{R}(H^{\dag}H^{k}),
$
and
$
\mathcal{N}(H^{C,\textcircled{\dag}}H)=\mathcal{N}(H^{\dag}P_{H^{k}}H)=\mathcal{N}(P_{H^{k}}H)=\mathcal{N}((H^{k})^{\dag}H)
$
that
$
H^{C,\textcircled{\dag}}H=P_{\mathcal{R}(H^{\dag}H^{k}),\mathcal{N}((H^{k})^{\dag}H)}.
$
\end{proof}
\begin{theorem}\label{thm3.9}
Let $H\in\mathbb{C}^{n\times n}_{k}$. Then the following assertions are equivalent:
\begin{itemize}
\item[(a)] $X=H^{C,\textcircled{\dag}}$;
\item[(b)] $\mathcal{R}(X^{*})=\mathcal{R}(H^{k})$ and $XH=H^{\dag}H_{1}$;
\item[(c)] $\mathcal{R}(X^{*})=\mathcal{R}(H^{k})$ and $XH^{k}=H^{\dag}H^{k}$;
\item[(d)] $\mathcal{R}(X^{*})=\mathcal{R}(H^{k})$ and $XH^{D}=H^{\dag}H^{D}$;
\item[(e)] $\mathcal{R}(X^{*})=\mathcal{R}(H^{k})$ and $XHH^{D}=H^{\dag,D}$.
\end{itemize}
\end{theorem}
\begin{proof}
``$(a)\Rightarrow(b)$''. The result is evident by Theorem \ref{thm3.1} and Corollary \ref{cor3.7}.

``$(b)\Rightarrow(c)$''. Note that $H_{1}H^{k-1}=P_{H^{k}}H^{k}=H^{k}.$
Postmultiplying $XH=H^{\dag}H_{1}$ by $H^{k-1}$, we obtain $XH^{k}=H^{\dag}H^{k}.$

``$(c)\Rightarrow(d)$''. Postmultiplying $XH^{k}=H^{\dag}H^{k}$ by $(H^{D})^{k+1}$, we get that $XH^{D}=H^{\dag}H^{D}$.

``$(d)\Rightarrow(e)$''. Postmultiplying $XH^{D}=H^{\dag}H^{D}$ by $H$, we have $XHH^{D}=H^{\dag,D}$.

``$(e)\Rightarrow(a)$''. Postmultiplying $XHH^{D}=H^{\dag,D}$ by $H^{k}$, we get that
$XH^{k}=H^{\dag}H^{k}.$ Postmultiplying $XH^{k}=H^{\dag}H^{k}$ by $(H^{k})^{\dag}$, we have $XP_{H^{k}}=H^{\dag}P_{H^{k}}$. It follows from $\mathcal{R}(X^{*})=\mathcal{R}(H^{k})$ that $X=XP_{H^{k}}$. Thanks to Theorem \ref{thm3.6}, we obtain $X=H^{\dag}P_{H^{k}}=H^{C,\textcircled{\dag}}$.
\end{proof}

\begin{theorem}
Let $H\in\mathbb{C}^{n\times n}_{k}$. Then the following assertions are equivalent:
\begin{itemize}
\item[(a)] $X=H^{C,\textcircled{\dag}};$
\item[(b)] $XH_{1}X=X, H_{1}X=H_{1}A^{\dag}$ and $XH_{1}=H^{\dag}H_{1}$;
\item[(c)] $r(X)=r(H^{k})$, $H_{1}X=H_{1}H^{\dag}$ and $XH_{1}=H^{\dag}H_{1}$;
\item[(d)] $r(X)=r(H^{k})$, $HX=H_{1}H^{\dag}$ and $XH=H^{\dag}H_{1}$;
\item[(e)] $XP_{H^{k}}=X$, $XH=H^{\dag}H_{1}$.
\end{itemize}
\end{theorem}
\begin{proof}
``$(a)\Rightarrow(b)$''. We need to check the following three equations:
$$
H^{C,\textcircled{\dag}}H_{1}H^{C,\textcircled{\dag}}=H^{C,\textcircled{\dag}},
H_{1}H^{C,\textcircled{\dag}}=(HH^{\textcircled{\dag}}H)H^{\dag},
H^{C,\textcircled{\dag}}H_{1}=H^{\dag}H_{1}.
$$

``$(b)\Rightarrow(c)$''. According to $X=XH_{1}X$, we derive that $$r(X)=r(H_{1}X)=r(H_{1}H^{\dag})=r(HH^{\textcircled{\dag}}HH^{\dag})=r(P_{H^{k}}P_{H})=r(P_{H^{k}})=r(H^{k}).$$

``$(c)\Rightarrow(a)$''. By $H_{1}X=H_{1}H^{\dag}=P_{H^{k}}$, we can obtain that $\mathcal{R}(H^{k})=\mathcal{R}(P_{H^{k}})=\mathcal{R}((P_{H^{k}})^{*})=\mathcal{R}((H_{1}X)^{*})\subseteq\mathcal{R}(X^{*}).$
Noting that $r(X)=r(H^{k})$ we get that $\mathcal{R}(X^{*})=\mathcal{R}(H^{k})$.
Postmultiplying $XH_{1}=H^{\dag}H_{1}$ by $H^{\textcircled{\dag}}$ and using $H_{1}H^{\textcircled{\dag}}=P_{H^{k}}$, we get that $XP_{H^{k}}=H^{\dag}P_{H^{k}}$. Postmultiplying
$XP_{H^{k}}=H^{\dag}P_{H^{k}}$ by $H^{k}$, we have $XH^{k}=H^{\dag}H^{k}$. Thus, due to (c) of Theorem \ref{thm3.9}, $X=H^{C,\textcircled{\dag}}$.

``$(a)\Rightarrow(d)$''. By Theorem \ref{thm3.1} and Corollary \ref{cor3.7}, the desired result follows.

``$(d)\Rightarrow(a)$''. The proof is analogous to that of $(c)\Rightarrow(a)$.

``$(a)\Rightarrow(e)$''. According to Theorem \ref{thm3.1} and Theorem \ref{thm3.6}, the desired result follows.

``$(e)\Rightarrow(a)$''. Note that $X=XP_{H^{k}}=(XH)(H^{k-1}(H^{k})^{\dag})=H^{\dag}(HH^{\textcircled{\dag}}H)(H^{k-1}(H^{k})^{\dag})=H^{\dag}P_{H^{k}}P_{H^{k}}=H^{\dag}P_{H^{k}}.$
Then $X=H^{C,\textcircled{\dag}}$.
\end{proof}

It is widely known that the generalized inverse can be expressed as a special kind of outer inverse along with prescribed range and null space. Therefore, we will prove that the same property holds in the case of CCE-inverse.

\begin{theorem}\label{thm3.11}
Let $H\in\mathbb{C}^{n\times n}_{k}$. Then
$$H^{C,\textcircled{\dag}}=H^{(2)}_{\mathcal{R}(H^{\dag}H^{k}),\mathcal{N}((H^{k})^{*})}.$$
\end{theorem}
\begin{proof}
From the definition of CCE-inverse, we know that $H^{C,\textcircled{\dag}}$ is an outer inverse of $H$. By Corollary \ref{cor3.7}, we have $$\mathcal{R}(H^{C,\textcircled{\dag}})=\mathcal{R}(H^{\dag}H^{k}),$$
$$\mathcal{N}(H^{C,\textcircled{\dag}})=\mathcal{N}((H^{k})^{*}).$$
So, we obtain this result.
\end{proof}

Next, we propose a connection between the $(B,C)-$inverse and CCE-inverse, which shows that the CCE-inverse of a matrix $H\in\mathbb{C}^{n\times n}_{k}$ is the $(H^{\dag}H^{k},(H^{k})^{*})-$inverse of $H$.
\begin{theorem}
Let $H\in\mathbb{C}^{n\times n}_{k}$. Then
$$H^{C,\textcircled{\dag}}=H^{(H^{\dag}H^{k},(H^{k})^{*})}.$$
\end{theorem}
\begin{proof}
Employing Corollary \ref{cor3.7}, we can obtain
$$
\mathcal{R}(H^{C,\textcircled{\dag}})=\mathcal{R}(H^{\dag}H^{k}), \ \ \ \mathcal{N}(H^{C,\textcircled{\dag}})=\mathcal{N}((H^{k})^{*}).
$$
Observe that
\begin{eqnarray*}
H^{C,\textcircled{\dag}}H(H^{\dag}H^{k})&=&H^{\dag}P_{H^{k}}H^{k}=H^{\dag}H^{k},\\
(H^{k})^{*}HH^{C,\textcircled{\dag}}&=&(H^{k})^{*}H(H^{\dag}HH^{\textcircled{\dag}}HH^{\dag})\\
&=&(H^{k})^{*}(HH^{\textcircled{\dag}})(HH^{\dag})=(H^{k})^{*}P_{H^{k}}\\
&=&(P_{H^{k}}H^{k})^{*}=(H^{k})^{*}.
\end{eqnarray*}
Thus, we obtain $H^{C,\textcircled{\dag}}=H^{(H^{\dag}P_{H^{k}},(H^{k})^{*})}.$
\end{proof}

\section{Two canonical forms of CCE-inverse and their applications}

In this section, we present two canonical forms of CCE-inverse by matrix decompositions, which is used to study the properties of CCE-inverse of $H$. Before that, we need to do some preparations.

For convenience, we adopt the following notation in the sequel.
\begin{itemize}
\item $\mathbb{C}^{\textrm{PI}}_{m,n}=\{H\mid H\in\mathbb{C}^{m\times n},  HH^{*}H=H\}$ denotes the set of comprising partial isometries;
\item Let $\mathbb{C}^{\textrm{P}}_{n}=\{H\mid H\in\mathbb{C}^{n\times n}, H^{2}=H\}$ be the set of projectors;
\item $\mathbb{C}^{\textrm{OP}}_{n}=\{H\mid H\in\mathbb{C}^{n\times n}, H^{2}=H=H^{*}\}$ is the set of orthogonal projectors;
\item $\mathbb{C}^{\textrm{EP}}_{n}=\{H\mid H\in\mathbb{C}^{n\times n}, HH^{\dag}=H^{\dag}H\}$ is the set of EP matrices;
\item $\mathbb{C}^{k,\dag}_{n}=\mathbb{C}^{k-\textrm{EP}}=\{H\mid H\in\mathbb{C}^{n\times n}_{k}, H^{k}H^{\dag}=H^{\dag}H^{k}\}$ denotes the set of $k$-EP matrices;
\item $\mathbb{C}^{i-\textrm{EP}}_{n}=\{H\mid H\in\mathbb{C}^{n\times n}_{k}, H^{k}(H^{k})^{\dag}=(H^{k})^{\dag}H^{k}\}$ denotes the set of $i$-EP matrices;
\item $\mathbb{C}^{k,\textcircled{\dag}}_{n}=\{H\mid H\in\mathbb{C}^{n\times n}_{k}, H^{k}H^{\textcircled{\dag}}=H^{\textcircled{\dag}}H^{k}\}$ is the set of $k$-Core EP matrices;
\item Let $\mathbb{C}^{k,D\dag}_{n}=\{H\mid H\in\mathbb{C}^{n\times n}_{k}, H^{k}H^{D,\dag}=H^{D,\dag}H^{k}\}$ be the set of $k$-DMP matrices;
\item Let $\mathbb{C}^{k,\dag D}_{n}=\{H\mid H\in\mathbb{C}^{n\times n}_{k},H^{k}H^{\dag,D}=H^{\dag,D}H^{k}\}$ be the set of dual $k$-DMP matrices;
\item $\mathbb{C}^{k,C\dag}_{n}=\{H\mid H\in\mathbb{C}^{n\times n}_{k}, H^{k}H^{C,\dag}=H^{C,\dag}H^{k}\}$ denotes the set of $k$-CMP matrices.
\end{itemize}

\subsection{The first canonical form derived from the Hartwig-Spindelbock decomposition}

On the basis of Corollary 6 in \cite{HS}, every $H\in\mathbb{C}^{n\times n}$ with rank $r$ can be expressed by
\begin{equation}\label{eq2.1}
H=U\left[\begin{array}{cc}
\Sigma M & \Sigma N \\
0 & 0 \\
\end{array}
\right]U^{*},
\end{equation}
where $U\in\mathbb{C}^{n\times n}$ is unitary, $\Sigma=diag(\sigma_{1},\sigma_{2},\ldots,\sigma_{r})$ is a diagonal matrix whose diagonal entries $\sigma_{i}>0$ $(i=1,2,\cdots,r)$ are singular values of $H$ and $M\in\mathbb{C}^{r\times r}$, $N\in\mathbb{C}^{r\times(n-r)}$ such that
\begin{equation}\label{eq2.2}
 MM^{*}+NN^{*}=I_{r}.
\end{equation}

According to the above decomposition, a straightforward computation shows that
 \begin{equation}\label{eq2.3}
H^{\dag}=U \left[\begin{array}{cc}
M^{*} \Sigma^{-1} & 0 \\
N^{*} \Sigma^{-1} & 0 \\
\end{array}
\right] U^{*},\ \ \
P_{H}=HH^{\dag}=U \left[\begin{array}{cc}
I_{r} & 0 \\
0 & 0 \\
\end{array}
\right] U^{*}.
\end{equation}

It is well-known that (see \cite{Bt2,flt2}):
 \begin{equation}\label{parameter}\nonumber
H^{\#}=U \left[\begin{array}{cc}
(\Sigma M)^{-1} & M^{-1}\Sigma^{-1}M^{-1}N\\
0 & 0 \\
\end{array}
\right] U^{*},\ \ \
H^{\textcircled{\#}}=U \left[\begin{array}{cc}
(\Sigma M)^{-1} & 0 \\
0 & 0 \\
\end{array}
\right] U^{*},
\end{equation}
\begin{equation}\label{eq2.4}
H^{\textcircled{\dag}}=U \left[\begin{array}{cc}
(\Sigma M)^{\textcircled{\dag}} & 0 \\
0 & 0 \\
\end{array}
\right] U^{*}.
\end{equation}

Using the decomposition of (\ref{eq2.1}), we have the following properties.

\begin{lemma}\label{lem2.1}{\rm \cite{Bs}}
Assume that $H\in\mathbb{C}^{n\times n}$ with rank $r$ is decomposed by $(\ref{eq2.1})$. Then,
\begin{itemize}
\item[(a)] $H\in\mathbb{C}^{\mathrm{CM}}_{n}\Leftrightarrow M$ is nonsingular;
\item[(b)] $H\in\mathbb{C}^{\mathrm{P}}_{n}\Leftrightarrow \Sigma M=I_{r}$;
\item[(c)] $H\in\mathbb{C}^{\mathrm{OP}}_{n}\Leftrightarrow N=0, \Sigma=M=I_{r}$;
\item[(d)] $H\in\mathbb{C}^{\mathrm{PI}}_{n,n}\Leftrightarrow \Sigma=I_{r}$;
\item[(e)] $H\in\mathbb{C}^{\mathrm{EP}}_{n}\Leftrightarrow N=0$.
\end{itemize}
\end{lemma}

Employing $(\ref{eq2.1})$, $(\ref{eq2.3})$ and $(\ref{eq2.4})$, the CCE-inverse can be represented by the following form.

\begin{theorem}\label{thm4.1}
Assume that $H\in\mathbb{C}^{n\times n}$ has the form of $(\ref{eq2.1})$ and $r(H)=r$. Then
\begin{equation}\label{eq4.1}
H^{C,\textcircled{\dag}}=U \left[\begin{array}{cc}
M^{*}M(\Sigma M)^{\textcircled{\dag}} & 0 \\
N^{*}M(\Sigma M)^{\textcircled{\dag}} & 0 \\
\end{array}
\right] U^{*}.
\end{equation}
\end{theorem}
\begin{proof}
By using $(\ref{eq2.1})$, $(\ref{eq2.3})$ and $(\ref{eq2.4})$, we get that
\begin{eqnarray*}
H^{C,\textcircled{\dag}}&=&H^{\dag}H_{1}H^{\dag}\\
&=&U \left[\begin{array}{cc}
M^{*}\Sigma^{-1} & 0 \\
N^{*}\Sigma^{-1} & 0 \\
\end{array}
\right]
\left[\begin{array}{cc}
\Sigma M(\Sigma M)^{\textcircled{\dag}} & 0 \\
0 & 0 \\
\end{array}
\right]
\left[\begin{array}{cc}
I_{r} & 0 \\
0 & 0 \\
\end{array}
\right]
U^{*}\\
&=&U\left[\begin{array}{cc}
M^{*}M(\Sigma M)^{\textcircled{\dag}} & 0 \\
N^{*}M(\Sigma M)^{\textcircled{\dag}} & 0 \\
\end{array}
\right]
 U^{*}.
\end{eqnarray*}
\end{proof}

By the expression of $H^{C,\textcircled{\dag}}$ in Theorem \ref{thm4.1}, we gather the following results.

\begin{theorem}
Let $H\in\mathbb{C}^{n\times n}$, $r(H)=r$. Then
\begin{itemize}
\item[(a)] $H^{C,\textcircled{\dag}}=0\Leftrightarrow H$ is a nilpotent matrix;
\item[(b)] $H^{C,\textcircled{\dag}}=P_{H}\Leftrightarrow H\in\mathbb{C}^{\mathrm{OP}}_{n}$;
\item[(c)] $H^{C,\textcircled{\dag}}=H\Leftrightarrow H\in\mathbb{C}^{\mathrm{EP}}_{n}$ and $H^{3}=H$;
\item[(d)] $H^{C,\textcircled{\dag}}=H^{\#}\Leftrightarrow H\in\mathbb{C}^{\mathrm{EP}}_{n}$;
\item[(e)] $H^{C,\textcircled{\dag}}=H^{\textcircled{\#}}\Leftrightarrow H\in\mathbb{C}^{\mathrm{EP}}_{n}$.
\end{itemize}
\end{theorem}
\begin{proof}
(a). By using the Theorem \ref{thm4.1}, we get that
$$H^{C,\textcircled{\dag}}=0\Leftrightarrow
M^{*}M(\Sigma M)^{\textcircled{\dag}}=0\ \ {\rm and} \ \
N^{*}M(\Sigma M)^{\textcircled{\dag}}=0.
$$
Combining them with $MM^{*}+NN^{*}=I_{r}$ leads to $ M(\Sigma M)^{\textcircled{\dag}}=0.$ Then $(\Sigma M)^{\textcircled{\dag}}=0,$ which means that
 $H^{\textcircled{\dag}}=0.$ Therefore, $H$ is a nilpotent matrix.

 Conversely, it is evident.

(b). ``$\Leftarrow$''. If $H\in\mathbb{C}^{\mathrm{OP}}_{n}$, then
$$H=U \left[\begin{array}{cc}
I_{r} & 0 \\
0 & 0 \\
\end{array}
\right]U^{*},$$
Thus, it is easy to see that $H^{C,\textcircled{\dag}}=P_{H}$.

``$\Rightarrow$''. Assume that $H^{C,\textcircled{\dag}}=P_{H}$. By Theorem \ref{thm4.1} and $(\ref{eq2.3})$, we get that
$$U \left[\begin{array}{cc}
M^{*}M(\Sigma M)^{\textcircled{\dag}}& 0 \\
N^{*}M(\Sigma M)^{\textcircled{\dag}}& 0 \\
\end{array}
\right]U^{*}=U \left[\begin{array}{cc}
I_{r}& 0 \\
0& 0 \\
\end{array}
\right]U^{*}.$$
Thus,
\begin{equation}\label{eq4.3}
M^{*}M(\Sigma M)^{\textcircled{\dag}}=I_{r},\ \ \ N^{*}M(\Sigma M)^{\textcircled{\dag}}=0.
\end{equation}
Due to $MM^{*}+NN^{*}=I_{r}$, it follows from $(\ref{eq4.3})$ that
$N=0,\;\ \Sigma M=I_{r}$.
Hence,
$$H=U \left[\begin{array}{cc}
I_{r}& 0 \\
0& 0 \\
\end{array}
\right]U^{*}\in\mathbb{C}^{\mathrm{OP}}_{n}.$$

(c). ``$\Leftarrow$''. If $H\in\mathbb{C}^{\mathrm{EP}}_{n}$ and $H^{3}=H$, by $(\ref{eq2.1})$ and Lemma \ref{lem2.1}, we get that
$$H=U \left[\begin{array}{cc}
\Sigma M& 0 \\
0& 0 \\
\end{array}
\right]U^{*}, \ \ \ (\Sigma M)^{2}=I_{r}.$$
Thus, $H^{C,\textcircled{\dag}}=H$ is evident.

``$\Rightarrow$''. Let $H^{C,\textcircled{\dag}}=H$. by $(\ref{eq2.1})$ and Theorem \ref{thm4.1}, we derive that
$$U \left[\begin{array}{cc}
M^{*}M(\Sigma M)^{\textcircled{\dag}}& 0 \\
N^{*}M(\Sigma M)^{\textcircled{\dag}}& 0 \\
\end{array}
\right]U^{*}=U \left[\begin{array}{cc}
\Sigma M& \Sigma N \\
0 & 0 \\
\end{array}
\right]U^{*},$$
which implies that
$$N=0, \ \ \ M^{*}M=I_{r}, \ \ \ (\Sigma M)^{\textcircled{\dag}}=(\Sigma M)^{-1}=\Sigma M.$$
Thus, we have $H^{3}=H$ and $H\in\mathbb{C}^{\mathrm{EP}}_{n}.$

(d). Using $(\ref{eq2.1})$, $(\ref{eq2.4})$ and Theorem \ref{thm4.1}, we show that $H^{C,\textcircled{\dag}}=H^{\#}$ if and only if $N=0$. Therefore, $H^{C,\textcircled{\dag}}=H^{\#}\Leftrightarrow H\in\mathbb{C}^{\mathrm{EP}}_{n}.$

(e). Using the same argument as in the proof of (d), the result can be easily carried out.
\end{proof}

\subsection{The second canonical form derived from the core-EP decompostion}

In Section 2, we have introduced the CEP-decomposition. In fact, several generalized inverses have the following expressions by utilizing the CEP-decomposition.

\begin{lemma}\label{lem2.2} {\rm \cite[Theorem 3.2]{w}}
 Let $H\in\mathbb{C}^{n\times n}_k$ be decomposed by $(\ref{eq2.5})$. Then $H^{\textcircled{\dag}}=H^{\textcircled{\#}}_{1}$. Furthermore
 \begin{equation}\label{eq2.6}
 H^{\textcircled{\dag}}=U \left[\begin{array}{cc}
T^{-1} & 0 \\
0 & 0 \\
\end{array}
\right] U^{*}.
\end{equation}
\end{lemma}
 \begin{lemma}\label{lem2.3}{\rm \cite{flt3,wd2}}
Let $H\in\mathbb{C}^{n\times n}_k$ be decomposed by $(\ref{eq2.5})$. Then
\begin{equation}\label{eq2.7}
H^{D}=U \left[\begin{array}{cc}
T^{-1} & (T^{k+1})^{-1}\widetilde{T} \\
0 & 0 \\
\end{array}
\right] U^{*},
\end{equation}
where $\widetilde{T}=\sum\limits_{j=0}^{k-1}T^{j}PQ^{k-1-j}$. Furthermore, $\widetilde{T}=0$ if and only if $P=0$.
\end{lemma}
 \begin{lemma}{\rm \cite{Dd}}
Let $H\in\mathbb{C}^{n\times n}_k$ have the form of $(\ref{eq2.5})$. Then
\begin{equation}\label{eq2.8}
H^{\dag}=U \left[\begin{array}{cc}
T^{*}\Lambda & -T^{*}\Lambda PQ^{\dag} \\
(I_{n-t}-Q^{\dag}Q)P^{*}\Lambda & Q^{\dag}-(I_{n-t}-Q^{\dag}Q)P^{*}\Lambda PQ^{\dag} \\
\end{array}
\right] U^{*},
\end{equation}
where $\Lambda=(TT^{*}+PP^{*}-PQ^{\dag}QP^{*})^{-1}$, $t=r(H^{k})$. Moreover,
\begin{equation}\label{eq2.9}
 HH^{\dag}=U \left[\begin{array}{cc}
I_{t} & 0 \\
0 & QQ^{\dag} \\
\end{array}
\right] U^{*},
\end{equation}
\begin{equation}
H^{\dag}H=U \left[\begin{array}{cc}
T^{*}\Lambda T & T^{*}\Lambda P(I_{n-t}-Q^{\dag}Q) \\
(I_{n-t}-Q^{\dag}Q)P^{*}\Lambda T & Q^{\dag}Q+(I_{n-t}-Q^{\dag}Q)P^{*}\Lambda P(I_{n-t}-Q^{\dag}Q) \\
\end{array}
\right] U^{*}.
\end{equation}
\end{lemma}

Utilizing the above expressions of generalized inverses, several characterizations of the sets consisting of special matrices were introduced as follows.

\begin{lemma}\label{lem2.5}{\rm \cite{flt3}}
Let $H\in\mathbb{C}^{n\times n}_{k}$ be given by $(\ref{eq2.5})$. Then
\begin{itemize}
\item[(a)] $H\in\mathbb{C}^{i-EP}_{n}\Leftrightarrow H\in\mathbb{C}^{k,\textcircled{\dag}}_{n}\Leftrightarrow P=0$;
\item[(b)] $H\in\mathbb{C}^{k-EP}_{n}\Leftrightarrow H\in\mathbb{C}^{k,C\dag}_{n}\Leftrightarrow P=PQ^{\dag}Q$ and $\widetilde{T}=\widetilde{T}QQ^{\dag}$, where $\widetilde{T}=\sum\limits_{j=0}^{k-1}T^{j}PQ^{k-j}$;
\item[(c)] $H\in\mathbb{C}^{k,D\dag}_{n}\Leftrightarrow \widetilde{T}=\widetilde{T}QQ^{\dag}$;
\item[(d)] $H\in\mathbb{C}^{k,\dag D}_{n}\Leftrightarrow P=PQ^{\dag}Q$;
\item[(e)] $H\in\mathbb{C}^{k,\textcircled{w}}_{n}\Leftrightarrow PQ=0$.
\end{itemize}
\end{lemma}

By Theorem \ref{thm3.6} and the CEP-decomposition, we derive another canonical form of the CCE-inverse.

\begin{theorem}\label{thm4.2}
Let $H\in\mathbb{C}^{n\times n}_{k}$ be expressed as in $(\ref{eq2.5})$. Then
\begin{equation}\label{eq4.2}
H^{C,\textcircled{\dag}}=U \left[\begin{array}{cc}
T^{*}\Lambda & 0 \\
(I_{n-t}-Q^{\dag}Q)P^{*}\Lambda & 0 \\
\end{array}
\right] U^{*},
\end{equation}
\end{theorem}
where $\Lambda=(TT^{*}+PP^{*}-PQ^{\dag}QP^{*})^{-1},\;\ t=r(H^{k}).$
\begin{proof}
By Theorem \ref{thm3.6} and $(\ref{eq2.8})$, we deduce that
\begin{eqnarray*}
H^{C,\textcircled{\dag}}&=&H^{\dag}P_{H^{k}}\\
&=&U \left[\begin{array}{cc}
T^{*}\Lambda & -T^{*}\Lambda PQ^{\dag} \\
(I_{n-t}-Q^{\dag}Q)P^{*}\Lambda & Q^{\dag}-(I_{n-t}-Q^{\dag}Q)P^{*}\Lambda PQ^{\dag} \\
\end{array}
\right] \left[\begin{array}{cc}
I_{t} & 0 \\
0 & 0 \\
\end{array}
\right] U^{*}\\
&=&U \left[\begin{array}{cc}
T^{*}\Lambda & 0 \\
(I_{n-t}-Q^{\dag}Q)P^{*}\Lambda & 0 \\
\end{array}
\right]U^{*},
\end{eqnarray*}
where $\Lambda=(TT^{*}+PP^{*}-PQ^{\dag}QP^{*})^{-1},\;\ t=r(H^{k})$.
\end{proof}

Now, some properties of $H^{C,\textcircled{\dag}}$ are deduced by Theorems \ref{thm4.2} in the following two theorems.

\begin{theorem}
Let $H\in\mathbb{C}^{n\times n}_{k}$ have the form of $(\ref{eq2.5})$. Then
\begin{itemize}
\item[(a)] $H^{C,\textcircled{\dag}}=H^{D}\Leftrightarrow P=0\Leftrightarrow H\in\mathbb{C}^{i-\mathrm{EP}}_{n}$ (or equivalently $H^{k}\in\mathbb{C}^{\mathrm{EP}}_{n}$);
\item[(b)] $H^{C,\textcircled{\dag}}=H^{\dag,D}\Leftrightarrow H\in\mathbb{C}^{i-\mathrm{EP}}_{n}$;
\item[(c)] $H^{C,\textcircled{\dag}}=H^{*}\Leftrightarrow H\in\mathbb{C}^{\mathrm{CM}}_{n}\cap\mathbb{C}^{\mathrm{PI}}_{n}$;
\item[(d)] $H^{C,\textcircled{\dag}}=H^{\textcircled{\dag}}\Leftrightarrow P=PQ^{\dag}Q\Leftrightarrow H\in\mathbb{C}^{k,\dag D}_{n}$(or equivalently $H^{k}H^{\dag D}=H^{\dag D}H^{k})$;
\item[(e)] $H^{C,\textcircled{\dag}}=H^{D,\dag}\Leftrightarrow H\in\mathbb{C}^{k,\textcircled{w}}_{n}\cap\mathbb{C}^{\textrm{k,\dag D}}_{n}$;
\item[(f)] $H^{C,\textcircled{\dag}}=H^{C,\dag}\Leftrightarrow P=PQ^{\dag}Q \Leftrightarrow H\in\mathbb{C}^{\textrm{k,\dag D}}_{n}$.
\end{itemize}
\end{theorem}
\begin{proof}
(a). By $(\ref{eq2.7})$ and $(\ref{eq4.2})$, we get that
\begin{eqnarray*}
H^{C,\textcircled{\dag}}=H^{D}&\Leftrightarrow& U \left[\begin{array}{cc}
T^{*}\Lambda & 0 \\
(I_{n-t}-Q^{\dag}Q)P^{*}\Lambda & 0 \\
\end{array}
\right]U^{*}=U \left[\begin{array}{cc}
T^{-1} & (T^{k+1})^{-1}\widetilde{T} \\
0 & 0 \\
\end{array}
\right]U^{*}\\
&\Leftrightarrow& T^{*}\Lambda=T^{-1},\ \ \ \widetilde{T}=0, \ \ \ P=PQ^{\dag}Q\\
&\Leftrightarrow& P=0.
\end{eqnarray*}

(b). Thanks to $(\ref{eq2.7})$ and  $(\ref{eq2.8})$, the canonical form of $H^{\dag D}$ is given by
\begin{equation} \label{eq4.4}
H^{\dag D}=H^{\dag}HH^{D}=U \left[\begin{array}{cc}
T^{*}\Lambda & T^{*}\Lambda T^{-k}\widetilde{T} \\
(I_{n-t}-Q^{\dag}Q)P^{*}\Lambda & (I_{n-t}-Q^{\dag}Q)P^{*}\Lambda T^{-k}\widetilde{T} \\
\end{array}
\right]U^{*}.
\end{equation}
Then, $H^{\dag D}=H^{C,\textcircled{\dag}}\Leftrightarrow \widetilde{T}=0$$\Leftrightarrow P=0$.

(c). It is easy to check that
\begin{eqnarray*}
H^{C,\textcircled{\dag}}=H^{*}&\Leftrightarrow& U \left[\begin{array}{cc}
T^{*}\Lambda & 0 \\
(I_{n-t}-Q^{\dag}Q)P^{*}\Lambda & 0 \\
\end{array}
\right]U^{*}=U \left[\begin{array}{cc}
T^{*} & 0 \\
P^{*} & Q^{*} \\
\end{array}
\right]U^{*}\\
&\Leftrightarrow& Q^{*}=0,\;\;\ T^{*}\Lambda=T^{*}\;\;\ {\rm and} \;\;\ (I_{n-t}-Q^{\dag}Q)P^{*}=P^{*} \\
&\Leftrightarrow& Q=0\;\;\;\ {\rm and}\;\;\;\ TT^{*}+PP^{*}=I_{t}\\
&\Leftrightarrow& H=U\left[\begin{array}{cc}
T & P \\
0 & 0 \\
\end{array}
\right]U^{*}\;\;\;\ {\rm and}\;\;\;\ HH^{*}H=H \\
&\Leftrightarrow& H\in\mathbb{C}^{\mathrm{CM}}_{n}\cap \mathbb{C}^{\mathrm{PI}}_{n}.
\end{eqnarray*}

(d). By $(\ref{eq2.6})$, $(\ref{eq4.2})$ and Lemma \ref{lem2.5}, a straightforward computation shows that
\begin{eqnarray*}
H^{C,\textcircled{\dag}}=H^{\textcircled{\dag}}&\Leftrightarrow& T^{*}\Lambda=T^{-1}\;\;\;\ {\rm and} \;\;\;\ (I_{n-t}-Q^{\dag}Q)P^{*}\Lambda=0 \\
&\Leftrightarrow& P=PQ^{\dag}Q\\
&\Leftrightarrow& H\in\mathbb{C}^{k,\dag D}_{n}.
\end{eqnarray*}

(e). It follows from $(\ref{eq2.7})$ and $(\ref{eq2.9})$ that
\begin{eqnarray*}
H^{D,\dag}&=&H^{D}HH^{\dag}\\
&=&U \left[\begin{array}{cc}
T^{-1} & (T^{k+1})^{-1}\widetilde{T} \\
0 & 0 \\
\end{array}
\right]\left[\begin{array}{cc}
I & 0 \\
0 & QQ^{\dag} \\
\end{array}
\right]U^{*}\\
&=&U \left[\begin{array}{cc}
T^{-1} & (T^{k+1})^{-1}\widetilde{T}QQ^{\dag} \\
0 & 0 \\
\end{array}
\right]U^{*}.
\end{eqnarray*}
According to $(\ref{eq4.2})$, we deduce that
\begin{eqnarray*}
H^{C,\textcircled{\dag}}=H^{D,\dag} &\Leftrightarrow& T^{-1}=T^{*}\Lambda,\;\;\ (I_{n-t}-Q^{\dag}Q)P^{*}\Lambda=0 \;\;\ and \;\;\ (T^{k+1})^{-1}\widetilde{T}QQ^{\dag}=0 \\
&\Leftrightarrow& P=PQ^{\dag}Q \;\;\ and \;\;\ \widetilde{T}QQ^{\dag}=0\\
&\Leftrightarrow& P=PQ^{\dag}Q \;\;\ and \;\;\ \widetilde{T}Q=0\\
&\Leftrightarrow& P=PQ^{\dag}Q \;\;\ and \;\;\ PQ=0\\
&\Leftrightarrow& H\in\mathbb{C}^{k,\textcircled{w}}_{n}\cap\mathbb{C}^{k,\dag D}_{n}(by\;\;\ Lemma \;\;\ \ref{lem2.5}).
\end{eqnarray*}

(f). By $(\ref{eq2.7})$, $(\ref{eq2.8})$ and  $(\ref{eq2.9})$, we obtain that
\begin{equation} \label{eq4.5}
H^{C,\dag}=U \left[\begin{array}{cc}
T^{*}\Lambda & T^{*}\Lambda T^{-k}\widetilde{T}QQ^{\dag} \\
(I_{n-t}-Q^{\dag}Q)P^{*}\Lambda & (I_{n-t}-Q^{\dag}Q)P^{*}\Lambda T^{-k}\widetilde{T}QQ^{\dag}\\
\end{array}
\right]U^{*}.
\end{equation}
Therefore, $H^{C,\textcircled{\dag}}=H^{C,\dag}\Leftrightarrow \widetilde{T}QQ^{\dag}=0\Leftrightarrow \widetilde{T}Q=0\Leftrightarrow PQ=0\Leftrightarrow H\in\mathbb{C}^{k,\textcircled{w}}_{n}$ (or equivalently $H^{k}H^{\textcircled{w}}=H^{\textcircled{w}}H^{k}$).
\end{proof}

\begin{theorem}
Let $H\in\mathbb{C}^{n\times n}_{k}$ have the form of $(\ref{eq2.5})$. Then the following assertions are equivalent:
\begin{itemize}
\item[(a)] $H^{C,\textcircled{\dag}}\in\mathbb{C}^{\mathrm{EP}}_{n}$;
\item[(b)] $P=PQ^{\dag}Q$;
\item[(c)] $H^{k}H^{\dag,D}=H^{\dag,D}H^{k}$;
\item[(d)] $H^{C,\textcircled{\dag}}H^{k+1}=H^{k}$;
\item[(e)] $H^{\dag}H^{k+1}=H^{k}$;
\item[(f)] $H^{C,\textcircled{\dag}}H^{k}=H^{\textcircled{w}}H^{k}$;
\item[(g)] $H^{C,\dag}H^{k}=H^{\textcircled{w}}H^{k}$;
\item[(h)] $HH^{D}=H^{\dag,D}H$;
\item[(i)] $HH^{D}=H^{C,\dag}H$.
\end{itemize}
\end{theorem}

\begin{proof}
``$(a)\Leftrightarrow(b)$''. By Theorem \ref{thm4.2}, we have
$$
H^{C,\textcircled{\dag}}=U \left[\begin{array}{cc}
T^{*}\Lambda & 0 \\
(I_{n-t}-Q^{\dag}Q)P^{*}\Lambda & 0 \\
\end{array}
\right] U^{*},
$$
where $\Lambda=[TT^{*}+P(I-Q^{\dag}Q)P^{*}]^{-1}.$
Therefore,
$$
H^{C,\textcircled{\dag}}\in\mathbb{C}^{\mathrm{EP}}_{n}\Leftrightarrow\mathcal{R}(H^{C,\textcircled{\dag}})
=\mathcal{R}(H^{(C,\textcircled{\dag})^{*}})\Leftrightarrow(I-Q^{\dag}Q)P^{*}\Lambda=0\Leftrightarrow P=PQ^{\dag}Q.
$$

``$(b)\Leftrightarrow(c)$''. The result can be proved by Theorem 3.16 of \cite{flt3}.

``$(b)\Leftrightarrow(d)$''. Due to $H=U \left[\begin{array}{cc}
T & P \\
0 & Q \\
\end{array}
\right] U^{*}$, we get that
$$
H^{k}=U \left[\begin{array}{cc}
T^{k} & \widetilde{T} \\
0 & 0 \\
\end{array}
\right] U^{*},
H^{k+1}=U \left[\begin{array}{cc}
T^{k+1} & T\widetilde{T} \\
0 & 0 \\
\end{array}
\right] U^{*},
$$
where $\widetilde{T}=T^{k-1}P+T^{k-1}PQ+\cdots+TPQ^{k-2}+PQ^{k-1}$. Thus, we have
$$H^{C,\textcircled{\dag}}H^{k+1}=H^{k}\Leftrightarrow \left[\begin{array}{cc}
T^{*}\Lambda & 0 \\
(I_{n-t}-Q^{\dag}Q)P^{*}\Lambda & 0 \\
\end{array}
\right]\left[\begin{array}{cc}
T^{k+1} & T\widetilde{T} \\
0 & 0 \\
\end{array}
\right]=\left[\begin{array}{cc}
T^{k} & \widetilde{T} \\
0 & 0 \\
\end{array}
\right]\Leftrightarrow P=PQ^{\dag}Q.$$

``$(e)\Leftrightarrow(d)$''. One can easily verify that $H^{C,\textcircled{\dag}}H^{k+1}=H^{k}\Leftrightarrow H^{\dag}P_{H^{k}}H^{k+1}=H^{k}\Leftrightarrow H^{\dag}H^{k+1}=H^{k}$.

``$(b)\Leftrightarrow(f)$''. Thanks to $H=U\left[\begin{array}{cc}
T & P\\
0 & Q \\
\end{array}
\right]U^{*},$ we obtain that $$H^{C,\textcircled{\dag}}=U\left[\begin{array}{cc}
T^{*}\Lambda & 0 \\
(I_{n-t}-Q^{\dag}Q)P^{*}\Lambda & 0 \\
\end{array}
\right]U^{*}, \ \ \
H^{\textcircled{w}}=U\left[\begin{array}{cc}
T^{-1} & T^{-2}P \\
0 & 0 \\
\end{array}
\right]U^*({\rm see Theorem 3.1 of cite{wc}}),
$$
$$
H^{k}=U\left[\begin{array}{cc}
T^{k} & \widetilde{T}\\
0 & 0 \\
\end{array}
\right]U^{*},
$$
where $\Lambda=[TT^{*}+P(I-Q^{\dag}Q)P^{*}]^{-1}$, $\widetilde{T}=T^{k-1}P+T^{k-2}PQ+\cdots+TPQ^{k-2}+PQ^{k-1}.$
Therefore, $H^{C,\textcircled{\dag}}H^{k}=H^{\textcircled{w}}H^{k}\Leftrightarrow(I-Q^{\dag}Q)P^{*}\Lambda T^{k}\Leftrightarrow P=PQ^{\dag}Q.$

``$(b)\Leftrightarrow(g)$''. By $(\ref{eq4.5})$, the proof is analogous to that of $(b)\Leftrightarrow(f)$.

``$(b)\Leftrightarrow(h)$''. It follows from $(\ref{eq2.7})$ and $(\ref{eq4.4})$ that
\begin{eqnarray*}
&&HH^{D}=H^{\dag,D}H\\
&\Leftrightarrow& \left[\begin{array}{cc}
T & P\\
0 & Q \\
\end{array}
\right]
\left[\begin{array}{cc}
T^{-1} & (T^{k+1})^{-1}\widetilde{T}\\
0 & 0 \\
\end{array}
\right]=\left[\begin{array}{cc}
T^{*}\Lambda & T^{*}\Lambda T^{-k}\widetilde{T}\\
(I_{n-t}-Q^{\dag}Q)P^{*}\Lambda & (I_{n-t}-Q^{\dag}Q)P^{*}\Lambda T^{-k}\widetilde{T} \\
\end{array}
\right]\left[\begin{array}{cc}
T & P\\
0 & Q \\
\end{array}
\right]\\
&\Leftrightarrow&\left[\begin{array}{cc}
I & T^{-k}\widetilde{T}\\
0 & 0 \\
\end{array}
\right]=\left[\begin{array}{cc}
T^{*}\Lambda T & T^{*}\Lambda P+T^{*}\Lambda T^{-k}\widetilde{T}Q\\
(I_{n-t}-Q^{\dag}Q)P^{*}\Lambda T & (I_{n-t}-Q^{\dag}Q)P^{*}\Lambda P + (I_{n-t}-Q^{\dag}Q)P^{*}\Lambda T^{-k}\widetilde{T}Q \\
\end{array}
\right]\\
&\Leftrightarrow&
\left\{\begin{array}{l}
I=T^{*}\Lambda T\\
T^{-k}\widetilde{T}=T^{*}\Lambda P+T^{*}\Lambda T^{-k}\widetilde{T}Q\\
0=(I_{n-t}-Q^{\dag}Q)P^{*}\Lambda T\\
0=(I_{n-t}-Q^{\dag}Q)P^{*}\Lambda P+(I_{n-t}-Q^{\dag}Q)P^{*}\Lambda T^{-k}\widetilde{T}Q.
\end{array}
\right.\\
&\Leftrightarrow& P=PQ^{\dag}Q.
\end{eqnarray*}

``$(h)\Leftrightarrow(i)$''. As $H^{C,\dag}H=H^{\dag,D}H$, then $HH^{D}=H^{C,\dag}H\Leftrightarrow HH^{D}=H^{\dag,D}H$.
\end{proof}

\section{CCE-inverse matrix and k-CCE matrix}

In this section, we generalize the definitions of the EP-matrix and $p$-EP matrix \cite{Mr} to the CCE-inverse matrix and k-CCE matrix, respectively. Furthermore, we investigate their characterizations by applying the CEP-decomposition. We start with the concept of CCE-inverse matrix.

\begin{definition}
Let $H\in\mathbb{C}^{n\times n}_{k}$. Then $H$ is called a CCE-inverse matrix if $HH^{C,\textcircled{\dag}}=H^{C,\textcircled{\dag}}H$.
\end{definition}
\begin{theorem}\label{thm5.2}
Let $H\in\mathbb{C}^{n\times n}$ have the form of $(\ref{eq2.5})$. Then the following assertions are equivalent:
\begin{itemize}
\item[(a)] $HH^{C,\textcircled{\dag}}=H^{C,\textcircled{\dag}}H$;
\item[(b)] $H_{1}H^{C,\textcircled{\dag}}=H^{C,\textcircled{\dag}}H_{1}$;
\item[(c)] $H_{1}H^{\dag}=H^{\dag}H_{1}$;
\item[(d)] $P=0$ (or equivalently $H^{k}\in\mathbb{C}^{\mathrm{EP}}_{n}$).
\end{itemize}
\end{theorem}
\begin{proof}
Since $HH^{\dag}H=H$ and $H^{\textcircled{\dag}}HH^{\textcircled{\dag}}=H^{\textcircled{\dag}}$, we get that $(a)\Leftrightarrow(b)\Leftrightarrow(c)$.

``$(c)\Rightarrow(d)$". Suppose that $H_{1}H^{\dag}=H^{\dag}H_{1}$. Then we have $HH^{\textcircled{\dag}}HH^{\dag}=H^{\dag}HH^{\textcircled{\dag}}H$.
Noting that $HH^{\textcircled{\dag}}HH^{\dag}=H^{\dag}HH^{\textcircled{\dag}}H$ if and only if $P_{H^{k}}=H^{\dag}P_{H^{k}}H$, we deduce that
$$ U\left[\begin{array}{cc}
I_{t} & 0 \\
0 & 0\\
\end{array}
\right]U^{*}=
U\left[\begin{array}{cc}
T^{*}\Lambda & -T^{*}\Lambda PQ^{\dag} \\
(I_{n-t}-Q^{\dag}Q)P^{*}\Lambda & Q^{\dag}-(I-Q^{\dag}Q)P^{*}\Lambda PQ^{\dag}\\
\end{array}
\right]\left[\begin{array}{cc}
I_{t} & 0 \\
0 & 0\\
\end{array}
\right]\left[\begin{array}{cc}
T & P \\
0 & Q\\
\end{array}
\right]
U^{*}.$$
Thus,
$$\left[\begin{array}{cc}
I_{t} & 0 \\
0 & 0\\
\end{array}
\right]=\left[\begin{array}{cc}
T^{*}\Lambda T & T^{*}\Lambda P \\
(I_{n-t}-Q^{\dag}Q)P^{*}\Lambda T & (I_{n-t}-Q^{\dag}Q)P^{*}\Lambda P\\
\end{array}
\right].$$
By the above equation, we have $P=0$ as $T^{*}\Lambda$ is nonsingular.

``$(d)\Rightarrow(c)$". Assume that $H=U\left[\begin{array}{cc}
T & 0 \\
0 & Q\\
\end{array}
\right]U^{*}$, which is from the CEP-decomposition $(\ref{eq2.5})$. Then we have
$$H_{1}=HH^{\textcircled{\dag}}H=U\left[\begin{array}{cc}
T & 0 \\
0 & 0\\
\end{array}
\right]U^{*}, \ \ \ H^{\dag}=U\left[\begin{array}{cc}
T^{-1} & 0 \\
0 & Q^{\dag}\\
\end{array}
\right]U^{*},$$
which implies that $H_{1}H^{\dag}=H^{\dag}H_{1}$.
\end{proof}

Define the matrix operation as $[E,F]=EF-FE$. Next, using the fact that $H$ is the CCE-inverse matrix if and only if $P=0$ (by Theorem \ref{thm5.2} (d)), some characterizations of the CCE-inverse matrix are derived as follows.
\begin{theorem}
Let $H\in\mathbb{C}^{n\times n}_{k}$ have the form of $(\ref{eq2.5})$. Then the following assertions are equivalent:
\begin{itemize}
\item[(a)] $H$ is a CCE-inverse matrix;
\item[(b)]  $H^{C,\textcircled{\dag}}=H^{D}$;
\item[(c)]  $H^{C,\textcircled{\dag}}=H^{\dag,D}$;
\item[(d)]  $H^{C,\textcircled{\dag}}=H^{\textcircled{w}}$;
\item[(e)]  $H^{k+1}H^{C,\textcircled{\dag}}=H^{k}$;
\item[(f)]  $H^{k+1}H^{\textcircled{\dag}}=H^{k}$;
\item[(g)]  $HH^{C,\textcircled{\dag}}=HH^{D}$;
\item[(h)]  $H^{k}H^{C,\textcircled{\dag}}=H^{k}H^{\textcircled{\dag}}$;
\item[(i)]  $[HH^{C,\textcircled{\dag}},H^{C,\textcircled{\dag}}H]=0$.
\end{itemize}
\end{theorem}
\begin{proof}
According to Theorem \ref{thm5.2}, we get that $H$ is a CCE-inverse matrix if and only if $P=0$. If $P=0$, then $H^{C,\textcircled{\dag}}=H^{D}=H^{\dag ,D}=H^{\textcircled{\dag}}=H^{\textcircled{w}}=U\left[\begin{array}{cc}
T^{-1} & 0 \\
0 & 0\\
\end{array}
\right]U^{*}.$ Thus it can be easily seen that $(a)\Rightarrow(b)$, $(a)\Rightarrow(c)$, $(a)\Rightarrow(d)$, $(a)\Rightarrow(e)$, $(a)\Rightarrow(f)$, $(a)\Rightarrow(g)$, $(a)\Rightarrow(h)$ and $(a)\Rightarrow(i)$.

``$(b)\Rightarrow(a)$". If $H^{C,\textcircled{\dag}}=H^{D}$, by $(\ref{eq2.7})$ and $(\ref{eq4.2})$, we have
$$
\left[\begin{array}{cc}
T^{*}\Lambda  & 0 \\
(I_{n-t}-Q^{\dag}Q)P^{*}\Lambda & 0\\
\end{array}
\right]=\left[\begin{array}{cc}
T^{-1} & (T^{k+1})^{-1}\widetilde{T} \\
0 & 0\\
\end{array}
\right]
,$$
then, $(T^{k+1})^{-1}\widetilde{T}=0\Rightarrow\widetilde{T}=0\Rightarrow P=0$.

``$(c)\Rightarrow(a)$". If $H^{C,\textcircled{\dag}}=H^{\dag,D}$, it follows from $(\ref{eq4.2})$ and $(\ref{eq4.4})$ that $\widetilde{T}=0\Rightarrow P=0$.

``$(d)\Rightarrow(a)$". If $H^{C,\textcircled{\dag}}=H^{\textcircled{w}}$, due to $(\ref{eq4.2})$ and Theorem 3.1 of \cite{wc}, we obtain
$$
\left[\begin{array}{cc}
T^{*}\Lambda  & 0 \\
(I_{n-t}-Q^{\dag}Q)P^{*}\Lambda & 0\\
\end{array}
\right]=\left[\begin{array}{cc}
T^{-1} & T^{-2}P \\
0 & 0\\
\end{array}
\right]
,$$
which implies $P=0$.

``$(e)\Rightarrow(a)$". According to $H^{k+1}H^{C,\textcircled{\dag}}=H^{k}$, we deduce that
$$
\left[\begin{array}{cc}
T^{k+1}  & T\widetilde{T} \\
0 & 0\\
\end{array}
\right]\left[\begin{array}{cc}
T^{*}\Lambda & 0 \\
(I_{n-t}-Q^{\dag}Q)P^{*}\Lambda & 0\\
\end{array}
\right]
=\left[\begin{array}{cc}
T^{k}  & \widetilde{T} \\
0 & 0\\
\end{array}
\right].$$
Thus, $\widetilde{T}=0$ which lead to $P=0$.

``$(f)\Rightarrow(a)$". The proof is similar to that of $(e)\Rightarrow(a)$.

 ``$(g)\Rightarrow(a)$". Assume that $HH^{C,\textcircled{\dag}}=HH^{D}$. Then $P_{H^{k}}=HH^{D}$ and
 $$P_{H^{k}}=HH^{D}\Rightarrow U\left[\begin{array}{cc}
I_{t}  & 0 \\
0 & 0\\
\end{array}
\right]U^{*}=U\left[\begin{array}{cc}
T  & P \\
0 & 0\\
\end{array}
\right]\left[\begin{array}{cc}
T^{-1}  & (T^{k+1})^{-1}\widetilde{T} \\
0 & 0\\
\end{array}
\right]U^{*}.$$
Thus, we get that $T^{-k}\widetilde{T}=0\Rightarrow \widetilde{T}=0\Rightarrow P=0$.

``$(h)\Rightarrow(a)$". The proof is similar to that of $(g)\Rightarrow(a)$.

``$(i)\Rightarrow(a)$". If $[HH^{C,\textcircled{\dag}},H^{C,\textcircled{\dag}}H]=0$, it follows from the fact that $HH^{C,\textcircled{\dag}}=P_{H^{k}}$ and $[HH^{C,\textcircled{\dag}},H^{C,\textcircled{\dag}}H]=0$ that
$$U\left[\begin{array}{cc}
I_{t}  & 0 \\
0 & 0\\
\end{array}
\right]\left[\begin{array}{cc}
T^{*}\Lambda  & 0 \\
(I_{n-t}-Q^{\dag}Q)P^{*}\Lambda & 0\\
\end{array}
\right]\left[\begin{array}{cc}
T  & P \\
0 & Q\\
\end{array}
\right]U^{*}
$$
$$
=U\left[\begin{array}{cc}
T^{*}\Lambda  & 0 \\
(I_{n-t}-Q^{\dag}Q)P^{*}\Lambda & 0\\
\end{array}
\right]\left[\begin{array}{cc}
T  & P \\
0 & Q\\
\end{array}
\right]\left[\begin{array}{cc}
I_{t}  & 0 \\
0 & 0\\
\end{array}
\right]U^{*}.$$
Thus,
$$
\left[\begin{array}{cc}
T^{*}\Lambda T  & T^{*}\Lambda P \\
0 & 0\\
\end{array}
\right]
=\left[\begin{array}{cc}
T^{*}\Lambda T  & 0 \\
(I_{n-t}-Q^{\dag}Q)P^{*}\Lambda T & 0\\
\end{array}
\right],$$
which implies that $P=0.$
\end{proof}

In [9], the authors introduced the following classes of matrices $\mathbb{C}^{k-\mathrm{EP}}_{n}$, $\mathbb{C}^{k,\textcircled{\dag}}_{n}$, $\mathbb{C}^{k,D\dag}_{n}$, $\mathbb{C}^{k,\dag D}_{n}$, $\mathbb{C}^{k,C\dag}_{n}$, $\mathbb{C}^{k,\textcircled{w}}_{n}$. Similarly,  we establish the concept of $k-$CCE-inverse matrix and propose the characterizations of $k-$CCE inverse matrix as follows.
\begin{definition}
Let $H\in\mathbb{C}^{n\times n}_{k}$. We say that $H$ is a k-CCE inverse matrix if $H\in\mathbb{C}^{k,C\textcircled{\dag}}_{n}$ and define the set of $k$-CCE inverse matrices by
$$\mathbb{C}^{k,C\textcircled{\dag}}_{n}=\{H\mid H\in\mathbb{C}^{n\times n}_{k}, H^{k}H^{C,\textcircled{\dag}}=H^{C,\textcircled{\dag}}H^{k}\}.$$
\end{definition}

Now, we give the characterizations of $k-$CCE-inverse matrices and prove that the $k-$CCE inverse matrix is the same as the CCE-inverse matrix.
\begin{theorem}
Let $H\in\mathbb{C}^{n\times n}_{k}$ have the form of $(\ref{eq2.5})$. Then the following assertions are equivalent:
\begin{itemize}
\item[(a)] $H\in\mathbb{C}^{k,C\textcircled{\dag}}$;
\item[(b)] $H$ is a CCE-inverse matrix;
\item[(c)] $H^{t}H^{C,\textcircled{\dag}}=H^{C,\textcircled{\dag}}H^{t}(t\geq k)$;
\item[(d)] $H^{m}_{1}H^{C,\textcircled{\dag}}=H^{C,\textcircled{\dag}}H^{m}_{1}(m\in \mathbb{N})$.
\end{itemize}
\end{theorem}
\begin{proof}
``$(a)\Leftrightarrow(b)$". By $(\ref{eq2.5})$, we have
$$H^{k}=U\left[\begin{array}{cc}
T^{k}  & \widetilde{T} \\
0 & 0\\
\end{array}
\right]U^{*},$$
where $\widetilde{T}=T^{k-1}P+T^{k-2}PQ+\cdots+TPQ^{k-2}+PQ^{k-1}$. Observe that
$$H^{C,\textcircled{\dag}}=U\left[\begin{array}{cc}
T^{*}\Lambda  & 0 \\
(I_{n-t}-Q^{\dag}Q)P^{*}\Lambda & 0\\
\end{array}
\right]U^{*},
$$
$$
H^{k}H^{C,\textcircled{\dag}}=H^{k}H^{\dag}P_{H^{k}}=H^{k-1}P_{H^{k}}=U\left[\begin{array}{cc}
T^{k-1}  & 0 \\
0 & 0\\
\end{array}
\right]U^{*}
,$$
$$
H^{C,\textcircled{\dag}}H^{k}=H^{\dag}P_{H^{k}}H^{k}=H^{\dag}H^{k}=
U\left[\begin{array}{cc}
T^{*}\Lambda T^{k}  & T^{*}\Lambda\widetilde{T} \\
(I_{n-t}-Q^{\dag}Q)P^{*}\Lambda T^{k}  & (I_{n-t}-Q^{\dag}Q)P^{*}\Lambda\widetilde{T}\\
\end{array}
\right]U^{*}.
$$
Hence, $H^{k}H^{C,\textcircled{\dag}}=H^{C,\textcircled{\dag}}H^{k}$ if and only if $\widetilde{T}=0$ and $(I_{n-t}-Q^{\dag}Q)P^{*}=0$, that is to say that $P=0$.

``$(b)\Leftrightarrow(c)$". The proof is analogous to that of $(a)\Leftrightarrow(b)$.

``$(b)\Leftrightarrow(d)$". Noting that $H_{1}=HH^{\textcircled{\dag}}H=U\left[\begin{array}{cc}
T  & P \\
0 & 0\\
\end{array}
\right]U^{*},$ we derive that
$$H^{m}_{1}=U\left[\begin{array}{cc}
T^{m}  & T^{m-1}P \\
0 & 0\\
\end{array}
\right]U^{*}, \ \ \ H_{1}^{m}H^{C,\textcircled{\dag}}=U\left[\begin{array}{cc}
T^{m}T^{*}\Lambda+T^{m-1}P(I_{n-t}-Q^{\dag}Q)P^{*}\Lambda  &  0\\
0 & 0\\
\end{array}
\right]U^{*},$$
$$H^{C,\textcircled{\dag}}H_{1}^{m}=U\left[\begin{array}{cc}
T^{*}\Lambda  & 0 \\
(I_{n-t}-Q^{\dag}Q)P^{*}\Lambda & 0\\
\end{array}
\right]\left[\begin{array}{cc}
T^{m}  & T^{m-1}P \\
0 & 0\\
\end{array}
\right]U^{*}\;\;\;\;\;\;\;\;\;\;\;\;\;\;\;\;\;\;\;\;\;\ $$
$$=U\left[\begin{array}{cc}
T^{*}\Lambda T^{m}  & T^{*}\Lambda T^{m-1}P \\
(I_{n-t}-Q^{\dag}Q)P^{*}\Lambda & (I_{n-t}-Q^{\dag}Q)P^{*}\Lambda T^{m-1}P\\
\end{array}
\right]U^{*},$$
where $\Lambda=(TT^{*}+P(I_{n-t}-Q^{\dag}Q)P^{*})^{-1}$, $t=r(H^{k}).$ Therefore, $H^{C,\textcircled{\dag}}H_{1}^{m}=H_{1}^{m}H^{C,\textcircled{\dag}}$ if and only if $P=0$.
\end{proof}

\section{Acknowledgements}\numberwithin{equation}{section}

The authors would like to appreciate Pro. Dragana of Cihu scholars in Hubei Normal university for her precious comments.


\end{document}